\documentclass{amsart}

\usepackage{amssymb,amsmath,latexsym,times}
\usepackage[foot]{amsaddr}
\usepackage{enumerate}
\usepackage{color}
\usepackage{graphicx}
\usepackage{url}
\usepackage{algorithm}
\usepackage{algorithmic}
\usepackage{tabularx}

\newcommand*\from{\colon}

\newtheorem{thm}{Theorem}

\newtheorem{prop}[thm]{Proposition}

\newtheorem{defin}[thm]{Definition}

\newcommand{\cl}{\mathrm{cl}}

\begin{document}

\title{Enumeration of $2$-Polymatroids on up to Seven Elements}
\author{Thomas J. Savitsky}
\address{Department of Mathematics\\
The George Washington University\\
Washington, DC  20052}
\email{savitsky@gwmail.gwu.edu}
\begin{abstract}A theory of single-element extensions of integer polymatroids
analogous to that of matroids is developed.  We present an algorithm to
generate a catalog of $2$-polymatroids, up to isomorphism.  When we
implemented this algorithm on a computer, obtaining all $2$-polymatroids
on at most seven elements, we discovered the surprising fact that
the number of $2$-polymatroids on seven elements fails to be unimodal in rank.
\end{abstract}
\maketitle

\thispagestyle{empty}


\section{Introduction}

A $k$-polymatroid is a generalization of a matroid in which
the rank of an element may be greater than $1$ but cannot exceed $k$.  
Precise definitions are given in the next section.
Polymatroids have applications in mathematics and computer science.  For example,
Chapter 11 of \cite{lovaszplummer} employs $2$-polymatroids in the study of
matching theory.  Polymatroids, and more generally, submodular functions, arise
in combinatorial optimization; see Part IV of \cite{schrijverB}.
We take the perspective that $k$-polymatroids are worth studying in their own right.

Although much work has been done with the use of computers on the enumeration
of small matroids, to our knowledge, none has been done on enumerating
$k$-polymatroids, where $k > 1$.  Some landmark results in matroid enumeration
include the following: in 1973, Blackburn, Crapo, and Higgs
\cite{eightels} published a catalog of all simple matroids on at most eight elements;
in 2008, Mayhew and Royle \cite{nineels} produced a catalog of all matroids
on up to nine elements; and in 2012, Matsumoto, Moriyama, Imai,
and Bremner \cite{matsumotoenum} enumerated all rank-$4$ matroids on
ten elements.  

In this paper, we describe our success in adapting the approach used by Mayhew
and Royle to $2$-polymatroids.  Using a desktop computer, we produced a catalog
of all $2$-polymatroids, up to isomorphism, on at most seven elements.  We
were surprised to discover that the number of $2$-polymatroids on 
seven elements is not unimodal in rank.

\section{Background}

For an introduction to polymatroids, we recommend Chapter 12 of
\cite{oxleybook}.  We begin our discussion with definitions.

\begin{defin}
Let $S$ be a finite set.  Suppose 
$\rho \from 2^S \to \mathbb{N}$ satisfies the following three conditions:
\begin{enumerate}[(i)]
\item if $X, Y \subseteq S$, then 
$\rho(X \cap Y) + \rho(X \cup Y) \le \rho(X) + \rho(Y)$
\emph{(submodular)},
\item if $X \subseteq Y \subseteq S$, then $\rho(X) \le \rho(Y)$
\emph{(monotone)}, and
\item $\rho(\varnothing) = 0$
\emph{(normalized)}.
\end{enumerate}
Then $(\rho, S)$ is termed an \emph{integer polymatroid} or simply a 
\emph{polymatroid} with \emph{rank function} $\rho$ and 
\emph{ground set} $S$.
\end {defin}

\begin{defin}
Let $k$ be a positive integer, and let $(\rho, S)$ be a polymatroid.  Suppose
that $\rho(x) \le k$ for every $x \in S$.  Then $(\rho, S)$ is a 
\emph{$k$-polymatroid}.  A \emph{matroid} may be defined as a $1$-polymatroid.
\end{defin}

Let $(\rho, S)$ and $(\tau, T)$ be polymatroids.  A function
$\sigma \from S \to T$ is an \emph{isomorphism} of polymatroids
if $\sigma$ is a bijection and if $\rho(X) = \tau(\sigma(X))$ for
every $X \subseteq S$.
The closure operator of a polymatroid may be defined exactly as that of a matroid.

\begin{defin}
The \emph{closure operator} $\cl:2^S\to 2^S$ of a
polymatroid $(\rho,S)$ is given by \\
$\cl_{\rho}(X) = \{x: \rho(X \cup x) = \rho(X) \}$ for $X\subseteq S$.
The set $\cl_{\rho}(X)$ is called the \emph{closure} of $X$ with respect to $\rho$.
The subscript is omitted when $\rho$ is clear from context.
\end{defin}

One can show that $\rho(X) = \rho(\cl(X))$ by induction on $|\cl(X) - X|$.
We will freely make use of this as well as the following  properties of
closure operators.  They are stated here without proof.

\begin{prop} The closure operator of a polymatroid $(\rho,S)$
satisfies the following three properties:
\begin{enumerate}[(i)]
\item $X\subseteq \cl(X)$ for all $X\subseteq S$
\emph{(increasing)},
\item if $X\subseteq Y\subseteq S$, then $\cl(X)\subseteq \cl(Y)$
\emph{(monotone)}, and
\item $\cl(X) = \cl(\cl(X))$ for all $X\subseteq S$
\emph{(idempotent)}.
\end{enumerate}
\end{prop}

A subset of the ground set that is maximal with respect to rank is called a
\emph{flat}.  Here is the definition in terms of the closure operator.

\begin{defin}
Let $(\rho, S)$ by a polymatroid. A set $X \subseteq S$ is called
a \emph{flat} of $\rho$ if $\cl(X) = X$.  The collection of flats
of $(\rho, S)$ is symbolized by $\mathcal{F}(\rho, S)$.
\end{defin}

Intersections of flats of matroids are themselves flats, and the same is 
true for polymatroids.

\begin{prop}If $F$ and $G$ are flats of polymatroid $(\rho, S)$, 
then $F \cap G$ is also a flat.
\end{prop}
\begin{proof}
Let $x \in S - (F \cap G)$.  Either $x \in S - F$ or 
$x \in S - G$.  By relabeling $F$ and $G$ if necessary,
we may assume $x \in S - F$.  By submodularity,
\begin{equation*}
\rho(F) + \rho((F \cap G) \cup x) \ge \rho(F \cup x) + \rho(F \cap G).
\end{equation*}
This implies $\rho((F \cap G) \cup x) - \rho(F \cap G) \ge
\rho(F \cup x) - \rho(F)$.  By assumption, $\rho(F \cup x) - \rho(F) > 0$, and
hence, as needed, $\rho((F\cap G)\cup x) - \rho(F\cap G)>0$.
\end{proof}

Since the entire ground set of a polymatroid is a flat, we 
see that the collection of flats of a polymatroid forms a 
lattice under set-inclusion.

The theory of single-element extensions of matroids was developed 
by Crapo in \cite{craposingleel}.  We extend this theory to polymatroids in the
next section, but first the matroid case is briefly reviewed here.  See 
Section 7.2 of \cite{oxleybook} for a detailed exposition.  We begin with
a couple of definitions that apply to polymatroids as well.

\begin{defin}
Let $(\rho, S)$ be a polymatroid, and let $e$ be an element
not in $S$.  If $(\bar{\rho},S\cup e)$
is a polymatroid with $\bar{\rho}(X)=\rho(X)$ for all $X\subseteq S$,
then $\bar{\rho}$ is a \emph{single-element extension} of $\rho$.
\end{defin}

\begin{defin}
A \emph{modular cut} of a polymatroid $(\rho,S)$ is a subset $\mathcal{M}$ of
$\mathcal{F}(\rho,S)$ for which
\begin{enumerate}[(i)]
\item if $F \in \mathcal{M}$, $G \in \mathcal{F}(\rho,S)$,
and $F \subseteq G$, then $G \in \mathcal{M}$, and
\item if $F, G \in \mathcal{M}$ and 
$\rho(F \cap G) + \rho(F \cup G) = \rho(F) + \rho(G)$,
then $F \cap G \in \mathcal{M}$.
\end{enumerate}
\end{defin}

The next two results show that single-element extensions of a matroid 
can be placed in one-to-one correspondence with its modular cuts.  This 
correspondence underlies the enumeration efforts in 
\cite{eightels} and \cite{nineels}. 

\begin{thm}\label{matroid_extA}
Suppose $(r, S)$ is a matroid with single-element extension
$(\bar{r}, S \cup e)$.  Define \\
$\mathcal{M} = \{F \in \mathcal{F}(r, S) : r(F) = \bar{r}(F \cup e)\}$.
Then $\mathcal{M}$ is a modular cut.
\end{thm}

\begin{thm}\label{matroid_extB}
Suppose $(r,S)$ is a matroid, $e$ is an element not in $S$, and
$\mathcal{M} \subseteq \mathcal{F}(r,S)$ is a modular cut.  Define $\bar{r}
\from 2^{S \cup e} \to \mathbb{N}$ as follows: for $X\subseteq S$,
set $\bar{r}(X) = r(X)$ and
\begin{equation*}
\bar{r}(X \cup e) =
\begin{cases}
    r(X) & \text{if} \,\, \cl(X) \in\mathcal{M},\\
    r(X) + 1 & \text{otherwise.}
\end{cases}
\end{equation*}
Then $(\bar{r}, S \cup e)$ is a matroid
and a single-element extension of $(r, S)$.
\end{thm}

Our final definition in this section will be used when we describe the flats of
single-element extensions.

\begin{defin}Let $F$ and $G$ be flats of a polymatroid $(\rho, S)$.
Suppose that $F \subsetneq G$ and that for any
flat $H$ with $F \subseteq H \subseteq G$, either $H = F$ or $H = G$.
Then we say that $G$ \emph{covers} $F$.
\end{defin}

\section{Single-Element Extensions of Polymatroids}

Given a polymatroid, our aim is to describe all of its single-element
extensions. As in the matroid case we may restrict our attention to flats
of the original polymatroid.  Suppose $(\bar{\rho}, S \cup e)$ is
a single-element extension of $(\rho, S)$.  The following proposition
shows that if the value of $\bar{\rho}(F \cup e)$ is known for every
flat $F$ of $(\rho, S)$, then $\bar{\rho}$ is completely determined.

\begin{prop}\label{propA}Suppose $(\bar{\rho}, S \cup e)$ is
a single-element extension of $(\rho, S)$.  Let $X \subseteq S$, and
let $\cl(X)$ be the closure of $X$ with respect to $\rho$ 
(not $\bar{\rho}$). Then 
$\bar{\rho}(X \cup e) = \bar{\rho}(\cl(X) \cup e)$.
\end{prop}
\begin{proof}
Since
$X\cup e\subseteq \cl(X)\cup e = \cl_{\bar{\rho}}(X)\cup e\subseteq
\cl_{\bar{\rho}}(X\cup e)$ and $\bar{\rho}$ has the same value on the
first and last of these sets, the result follows.
\end{proof}

For a single-element extension $(\bar{\rho},S\cup e)$
of $(\rho,S)$, let $c$ be $\bar{\rho}(e)$
and let $X \subseteq S$.  It follows that $\bar{\rho}(X \cup e) \le
\rho(X) + c$ by the submodularity and normalization of
$\bar{\rho}$. Therefore, we may partition the flats of $(\rho, S)$ into
classes $\mathcal{M}_0, \mathcal{M}_1, \ldots, \mathcal{M}_c$ by the 
rule $F \in \mathcal{M}_i$ if and only if
$\bar{\rho}(F \cup e) = \rho(F) + i$.  (Note that some $\mathcal{M}_i$
may be empty.)  By Proposition~\ref{propA}, knowledge of 
$(\rho, S)$ and the partition 
$(\mathcal{M}_0, \mathcal{M}_1, \ldots, \mathcal{M}_c)$ completely
determines $(\bar{\rho}, S \cup e)$.  Our goal is to develop properties
that characterize such partitions. The following definition will be useful.

\begin{defin}Let $(\rho, S)$ be a polymatroid, and let $X, Y \subseteq S$.
Define the \emph{modular defect} of $X$ and $Y$,
denoted $\delta(X, Y)$, to be
$\rho(X) + \rho(Y) - \rho(X \cup Y) - \rho(X \cap Y)$.  If $\delta(X, Y) = 0$,
then $X$ and $Y$ are a \emph{modular pair} of sets.
\end{defin}

Now suppose $(\mathcal{M}_0, \mathcal{M}_1, \ldots, \mathcal{M}_c)$ is a
partition of $\mathcal{F}(\rho, S)$.  Let $e$ be an element
not in $S$ and define $\bar{\rho} \from 2^{S\cup e}\to \mathbb{N}$ as
follows: for $X\subseteq S$, set $\bar{\rho}(X)=\rho(X)$ and, if
$\cl(X)\in \mathcal{M}_i$, then set $\bar{\rho}(X\cup e)=\rho(X)+i$.
Furthermore, define a function $\mathcal{\mu} \from 2^{S} \to \mathbb{N}$ by
$\mathcal{\mu}(X) = i$ if $\cl(X) \in \mathcal{M}_i$.

\begin{thm}\label{thmB}As defined above, $(\bar{\rho}, S \cup e)$
is a polymatroid, and hence a single-element extension of $(\rho, S)$,
if and only if the following three conditions hold for all flats $F, G$ 
of $(\rho, S)$:
\begin{enumerate}[(I)]
\item $\mu(F \cap G) + \mu(F \cup G) - \delta(F, G) \le
\mu(F) + \mu(G)$,
\item if $F \subseteq G$, then 
$\rho(F) + \mu(F) \leq \rho(G) + \mu(G)$, and
\item if $F \subseteq G$, then $\mu(G) \le \mu(F)$.
\end{enumerate}
\end{thm}
\begin{proof}
Assume $(\bar{\rho}, S \cup e$) is a polymatroid,
and let $F, G$ be flats of $(\rho, S)$.
Applying the submodularity of $\bar{\rho}$ to the pair of sets $F \cup e$
and $G \cup e$ gives
\begin{equation*}
\bar{\rho}((F \cup e) \cap (G \cup e)) + \bar{\rho}((F \cup e) \cup (G \cup e))
\le \bar{\rho}(F \cup e) + \bar{\rho}(G \cup e).
\end{equation*}
By our definition of $\bar{\rho}$, the right side of the above
inequality equals $\rho(F) + \mu(F) + \rho(G) + \mu(G)$.  The
left side equals
\begin{equation*}
\begin{split}
\bar{\rho}((F \cap G) \cup e) + \bar{\rho}((F \cup G) \cup e) 
& = \rho(F \cap G) + \mu(F \cap G) + 
    \rho(F \cup G) + \mu(F \cup G) \\
& = \mu(F \cap G) + \mu(F \cup G) + \rho(F) + \rho(G) - \delta(F, G).
\end{split}
\end{equation*}

\noindent
We conclude that $\mu(F \cap G) + \mu(F \cup G) -\delta(F, G) 
\le \mu(F) + \mu(G)$ and see that
condition (I) is satisfied.

Statement (II) is the monotone property of $\bar{\rho}$.

Finally, to show condition (III), apply the submodularity of $\bar{\rho}$ to
the pair of sets $F \cup e$ and $G$.  This gives the first of the following
equivalent inequalities:
\begin{enumerate}
\item $\bar{\rho}((F \cup e) \cup G) + \bar{\rho}((F \cup e) \cap G) \le \bar{\rho}(F \cup e) + \bar{\rho}(G)$
\item $\bar{\rho}(G \cup e) + \bar{\rho}(F) \le \bar{\rho}(F \cup e) + \bar{\rho}(G)$
\item $\bar{\rho}(G \cup e) - \bar{\rho}(G) \le \bar{\rho}(F \cup e) - \bar{\rho}(F)$
\item $\mu(G) \le \mu(F)$.
\end{enumerate}

Now assume that conditions (I), (II), and (III) are satisfied.
We must verify that $\bar{\rho}$ satisfies the three axioms for a 
polymatroid.  It follows immediately from our definition that
$\bar{\rho}(\varnothing) = 0$.

Next, we check monotonicity.  Assume that $X \subseteq Y \subseteq S$.  
The definition of $\bar{\rho}$ and the monotonicity of $\rho$ imply that 
$\bar{\rho}(X) = \rho(X) \le \rho(Y) = \bar{\rho}(Y)$.
Thus we also get
$\bar{\rho}(X)\leq\bar{\rho}(Y)\leq\bar{\rho}(Y\cup e)$.
It remains to check that $\bar{\rho}(X \cup e) \le \bar{\rho}(Y \cup e)$.
Observe
\begin{equation*}
\begin{split}
\bar{\rho}(X \cup e) 
& = \rho(X) + \mu(X) \\
& = \rho(\cl(X)) + \mu(\cl(X)) \\
& \le \rho(\cl(Y)) + \mu(\cl(Y)) \qquad \text{(by condition (II))} \\
& = \rho(Y) + \mu(Y) \\
& = \bar{\rho}(Y \cup e). \\
\end{split}
\end{equation*}
Therefore, $\bar{\rho}$ is monotone on all subsets of $S \cup e$.

Since $\bar{\rho}(X) = \rho(X)$ for $X\subseteq S$, to check submodularity
it suffices to verify it for the pairs (a) $X\cup e$ and $Y$, 
and (b) $X\cup e$ and $Y\cup e$, with $X,Y\subseteq S$.  For case (a), we have
\begin{equation*}
\begin{split}
\bar{\rho}((X \cup e) \cap Y) + \bar{\rho}((X \cup e) \cup Y) 
& = \bar{\rho}(X \cap Y) + \bar{\rho}((X \cup Y) \cup e) \\
& = \rho(X \cap Y) + \rho(X \cup Y) + \mu(\cl(X \cup Y)) \\
& \le \rho(X) + \rho(Y) + \mu(\cl(X \cup Y)) \qquad \text{(by the submodularity of $\rho$)} \\
& \le \rho(X) + \rho(Y) + \mu(\cl(X)) \qquad \text{(by condition (III))} \\
& = \bar{\rho}(X \cup e) + \bar{\rho}(Y).
\end{split}
\end{equation*}
For case (b), we have
\begin{equation*}
\begin{split}
\bar{\rho}(X \cup e) + \bar{\rho}(Y \cup e) 
& = \rho(\cl(X)) + \mu(\cl(X)) + \rho(\cl(Y)) + \mu(\cl(Y)) \\
& \ge \mu(\cl(X) \cap \cl(Y)) + \mu(\cl(X) \cup \cl(Y)) - \delta(\cl(X), \cl(Y)) + \rho(\cl(X)) + \rho(\cl(Y))  \\
& = \mu(\cl(X) \cap \cl(Y)) + \mu(\cl(X) \cup \cl(Y)) + \rho(\cl(X) \cup \cl(Y)) + \rho(\cl(X) \cap \cl(Y)) \\
& = \bar{\rho}((\cl(X) \cup \cl(Y)) \cup e) + \bar{\rho}((\cl(X) \cap \cl(Y)) \cup e) \\
& \ge \bar{\rho}(X \cup Y \cup e) + \bar{\rho}((X \cap Y) \cup e).\\
\end{split}
\end{equation*}
\noindent
The first inequality follows by condition (I), and the last inequality holds 
because the monotonicity of $\bar{\rho}$ has already been established.
\end{proof}

Note that Theorem~\ref{thmB} generalizes Theorems~\ref{matroid_extA} and
\ref{matroid_extB} for single-element extensions
of matroids.  Also note that if the conditions of the
theorem are satisfied, then $\mathcal{M}_0$ is a modular cut.  Lastly, we point
out that the theorem remains true if the word ``flats'' is replaced by ``sets''
in its statement.  

\begin{defin}A partition $(\mathcal{M}_0, \mathcal{M}_1, \ldots, \mathcal{M}_c)$
of flats of a polymatroid $(\rho, S)$ that satisfies the conditions in
Theorem~\ref{thmB} is called an \emph{extensible partition.}
\end{defin}

For the remainder of this section, assume that
$(\rho, S)$ is a polymatroid, $(\mathcal{M}_0, \mathcal{M}_1, \ldots, \mathcal{M}_c)$
is an extensible partition, and $(\bar{\rho}, S \cup e)$ is the single-element
extension defined right before Theorem~\ref{thmB}.  Our next
goal is to describe the flats of $(\bar{\rho}, S \cup e)$.

Clearly if $F \subseteq S$ is a flat of $(\bar{\rho}, S \cup e)$, then $F$
is also a flat of $(\rho, S)$.  We also have the following helpful fact.

\begin{prop}For $F\subseteq S$, if $F\cup e$ is a flat of
$(\bar{\rho},S \cup e)$, then $F$ is a flat of $(\rho,S)$.
\end{prop}
\begin{proof}
Observe that $\cl_\rho(F)\subseteq \cl_{\bar{\rho}}(F)\subseteq
\cl_{\bar{\rho}}(F\cup e)=F\cup e$.
\end{proof}

Therefore, to find the flats of $(\bar{\rho}, S \cup e)$ we need only
consider sets of the form $F$ and $F \cup e$, where $F$ is a flat of
$(\rho, S)$.  The next proposition explicitly describes the flats of
$(\bar{\rho}, S \cup e)$.

\begin{prop}Let $(\bar{\rho}, S \cup e)$ be the single-element extension
of $(\rho, S)$ corresponding to the extensible partition 
$(\mathcal{M}_0, \mathcal{M}_1, \ldots, \mathcal{M}_c)$.  The flats
of $(\bar{\rho}, S \cup e)$ are the sets
\begin{enumerate}
\item $F$ in $\mathcal{M}_i$, for $i>0$,
\item $F\cup e$, for $F\in\mathcal{M}_0$, 
\item $F\cup e$, for $F\in \mathcal{M}_i$ with $i>0$, where
$F$ has no cover $G$ with
$\rho(F) + \mu(F) = \rho(G) + \mu(G)$.
\end{enumerate}
\end{prop}
\begin{proof}
To reiterate, we need only look at sets of the form $F$ and $F \cup e$,
where $F$ is a flat of $(\rho, S)$.

It follows from the
definition of $\bar{\rho}$ that a flat $F$ of $(\rho, S)$ is a flat of
$(\bar{\rho},S \cup e)$ if and only if $F \not\in \mathcal{M}_0$.

If $F\in\mathcal{M}_0$, then $F\cup e$ is a flat of $(\bar{\rho},S\cup
e)$ since, for $y\in S-F$, we have
$$\bar{\rho}(F\cup \{e, y\}) \geq \rho(F\cup y)>\rho(F)
=\bar{\rho}(F\cup e).$$

We claim that for $F\in \mathcal{M}_i$ with $i>0$, the set $F\cup e$
is a flat of $(\bar{\rho},S\cup e)$ if and only if the inequality in
property (II) of Theorem~\ref{thmB} is strict for all covers $G$ of $F$.
Indeed, if $G$ covers $F$ and $\rho(F) + \mu(F) = \rho(G) + \mu(G)$, then
$F\cup e$ is not a flat since
$$\bar{\rho}(F\cup e)= \rho(F)+\mu(F) =  \rho(G)+\mu(G)= 
\bar{\rho}(G\cup e).$$ Now assume strict inequality holds in property
(II) for all covers of $F$.  If $x\in S-F$, then there is a cover $G$
of $F$ with $F\subsetneq G\subseteq \cl_{\rho}(F\cup x)$,
so
\begin{equation*}
\bar{\rho}(F\cup e)= \rho(F)+\mu(F) <
\rho(G)+\mu(G)=\bar{\rho}(G\cup e) \leq \bar{\rho}(F\cup \{e,x\}).
\qedhere
\end{equation*}
\end{proof}

Note that if $\mu(G) = \mu(F)$, then equality cannot hold in
property (II) of Theorem~\ref{thmB}, since $\rho(G) > \rho(F)$.

These results generalize those for matroid extension.
We define the \emph{collar}
of $\mathcal{M}_i$ to consist of every $F \in \mathcal{M}_i$
that is covered by some $G \in \mathcal{M}_j$ with $j < i$.
In a matroid $(r, S)$,
if a flat $G$ covers a flat $F$, then $r(G) - r(F) = 1$.
If $(\bar{r}, S \cup e)$ is a single-element extension and
$F \in \mathcal{M}_1$, then $F \cup e$ is
a flat of $\bar{r}$ if and only if $F$ is not in the collar
of $\mathcal{M}_1$.

\section{Generating a Catalog of Small $2$-Polymatroids}

Now we will specialize the results of the previous section to
$2$-polymatroids.

Suppose $(\rho, S)$ is a $2$-polymatroid with collection of flats $\mathcal{F}(S)$.
Suppose that $\mathcal{F}(S)$ is the union of three disjoint 
sets, $\mathcal{M}_0$, $\mathcal{M}_1$, and $\mathcal{M}_2$, some of which may
be empty.  Let $e$ be an element not in $S$.  We define a
function $\bar{\rho} \colon 2^{S \cup e} \to \mathbb{N}$ as follows.
For $X \subseteq S$, define $\bar\rho(X) = \rho(X)$ and
$$\bar{\rho}(X \cup e) = 
\rho(X)+i \,\,\text{where}\,\, \cl(X)\in \mathcal{M}_i.$$

When computing the extensible partitions of a $2$-polymatroid, we found
it convenient to work with the following verbose specialization
of Theorem~\ref{thmB}.  

\begin{thm}\label{thmC}As defined, $(\bar\rho, S \cup e)$ is a $2$-polymatroid 
extension of $(\rho, S)$ if and only if the following seven
conditions are met.
\begin{enumerate}[(1)]
\addtolength{\itemsep}{5pt}
\item If $F \in \mathcal{M}_2$, $G \in \mathcal{F}(S)$, $F \subseteq G$, and
$\rho(G) - \rho(F) = 1$, then $G \in \mathcal{M}_1 \cup
\mathcal{M}_2$.  In other words, if $F \in \mathcal{M}_2$ is
covered by a flat $G$ of $(\rho, S)$ of one rank higher, then
$G$ cannot be in $\mathcal{M}_0$.
\item If $F, G \in \mathcal{M}_0$ and $(F, G)$ is a modular
pair, then $F \cap G \in \mathcal{M}_0$ as well.
\item If $F, G \in \mathcal{M}_0$ and $\rho(F) + \rho(G) = 
\rho(F \cup G) + \rho(F \cap G) + 1$, then $F \cap G \in \mathcal{M}_0
\cup \mathcal{M}_1$.
\item If $F, G \in \mathcal{M}_1$ and $(F, G)$ is a modular
pair, then either $F \cap G \in \mathcal{M}_1$ as well, or 
$F \cap G \in \mathcal{M}_2$ and $\cl(F \cup G) \in \mathcal{M}_0$.
\item If $F \in \mathcal{M}_0$, $G \in \mathcal{M}_1$, and $(F, G)$
is a modular pair, then $F \cap G$ cannot be in $\mathcal{M}_2$.
\item The set $\mathcal{M}_2$ is down-closed in the lattice $\mathcal{F}(\rho, S)$.
\item The set $\mathcal{M}_0$ is up-closed in the lattice $\mathcal{F}(\rho, S)$.
\end{enumerate}
\end{thm}
\begin{proof}[Sketch of Proof] Condition (II) of Theorem~\ref{thmB}
specializes to condition (1) here, 
condition (I) to conditions (2) through (5), and
condition (III) to conditions (6) and (7).
\end{proof}

The flats of $(\bar{\rho},S\cup e)$ are the sets
\begin{enumerate}
\item $F$ in $\mathcal{M}_1\cup \mathcal{M}_2$,
\item $F\cup e$, for $F\in\mathcal{M}_0$, 
\item $F\cup e$, for $F\in \mathcal{M}_i$, with $i>0$, where,
\begin{enumerate}
\item $F$ has no cover $G$ in $\mathcal{M}_{i-1}$ with
  $\rho(G)=\rho(F)+1$, and
\item if $i=2$, $F$ has no cover $G$ in $\mathcal{M}_0$ with
  $\rho(G)=\rho(F)+2$.
\end{enumerate}
\end{enumerate}

For example, let
$(\rho, \{a, b\})$ be the $2$-polymatroid consisting of two lines
placed freely in a plane.  To be specific, define $\rho(\varnothing) = 0$,
$\rho(\{a\}) = \rho(\{b\}) =2$, and $\rho(\{a, b\}) = 3$. The single-element
extension corresponding to the extensible partition 
$$(\mathcal{M}_0, \mathcal{M}_1, \mathcal{M}_2) = 
(\{\{a, b\}\}, \{\{a\}, \{b\}\}, \{\varnothing\})$$ is the $2$-polymatroid
consisting of three lines placed freely in a plane.

Using the results of this section,
we endeavored to catalog all small $2$-polymatroids on a computer
by means of a canonical deletion algorithm.

\begin{defin}
Suppose $\mathcal{X}$ is a collection of combinatorial objects with
ground set $\{1, \ldots, n\}$ and a notion of isomorphism.  
A function $C \from \mathcal{X} \to \mathcal{X}$
is a \emph{canonical labeling} function if the following hold
for all $X, Y \in \mathcal{X}$:
\begin{enumerate}[(i)]
\item $X$ is isomorphic to $C(X)$, and
\item $C(X) = C(Y)$ if and only if $X$ is isomorphic to $Y$.
\end{enumerate}
In this case, $C(X)$ is called the
\emph{canonical representative} of $X$.
\end{defin}

Brendan McKay's \texttt{nauty} program efficiently computes canonically
labelings of colored graphs.  In order to make use of it, we convert
polymatroids into graphs using the following construction.

\begin{defin}
Given an integer polymatroid $(\rho, S)$, define a colored, bipartite graph
with bipartition $S$ and $\mathcal{F}(\rho, S)$.  An edge between $e \in S$ and 
$F \in \mathcal{F}(\rho, S)$ exists if and only if $e \in F$.
Color $F \in \mathcal{F}(\rho, S)$ with its rank, $\rho(F)$.
Color each $e \in S$ with $-1$.
Call the resulting graph the \emph{flat graph}
\footnote{In our implementation, we found it prudent to insert
an isolated vertex of rank $r$ if no flats of rank $r$ existed, for 
$r < \rho(S)$.  This made it easier to work with the labelings used by
\texttt{nauty}.}
of the integer polymatroid.
\end{defin}

Note that if $X \subseteq S$ and if $F$ is the smallest flat containing $X$,
then $\rho(X) = \rho(F)$.  In terms of the flat graph, the rank of a set
$X \subseteq S$ equals the least color amongst those vertices adjacent to every
element of $X$.  Using this observation, it is easy to prove the next 
proposition.

\begin{prop}Two integer polymatroids are isomorphic if and only if their
flat graphs are isomorphic as colored graphs. (By an isomorphism of a
colored graph, we mean a graph isomorphism that maps
each vertex to another of the same color.)
\end{prop}

Therefore, in order to canonically label a $2$-polymatroid, it suffices
to consider its flat graph.  Then \texttt{nauty} is
used to compute a canonical labeling of the flat graph.  When restricted to the
ground set of the polymatroid, this gives a canonical labeling of the polymatroid.
For a description of the algorithms used by \texttt{nauty} see \cite{pgi1} and \cite{pgi2}.  
One may also find the exposition in \cite{mckayexposition} helpful.

Now we have all the tools needed to adapt Algorithm 1 of \cite{nineels} to
$2$-polymatroids.  Suppose we are given a set $X_n$ that consists of precisely one
representative of each isomorphism class of $2$-polymatroids on the ground set
$\{1, \ldots, n\}$.  The following algorithm
produces its counterpart, $X_{n+1}$, for the set $\{1,\ldots,n+1\}$.

\begin{algorithm}[H]
\caption{Isomorph-free generation of $2$-polymatroids}
\label{kpolymalg1}
\begin{algorithmic}
    \FOR{\textbf{each} $\rho \in X_n$}
        \STATE Set $Y_{\rho} \gets \varnothing$, the collection of extensions
               of $\rho$ that should appear in $X_{n+1}$.
            \FOR{\textbf{each} extensible partition $(\mathcal{M}_0, \mathcal{M}_1, \mathcal{M}_2)$ of $\rho$}
                \STATE Let $\bar{\rho}$ be the extension of $\rho$ associated with this partition.
                \STATE Canonically label $\bar{\rho}$.
                \STATE Set $\rho' \gets \bar{\rho} \, \backslash (n+1)$, the canonical deletion.
                \STATE Canonically label $\rho'$.
                \IF {$\rho = \rho'$ \AND $\bar{\rho} \not\in Y_{\rho}$}
                      \STATE Set $Y_{\rho} \gets Y_{\rho} \cup \bar{\rho}$.
                \ENDIF
            \ENDFOR
        \STATE Set $X_{n+1} \gets X_{n+1} \cup Y_{\rho}$.
    \ENDFOR
    \RETURN $X_{n+1}$
\end{algorithmic}
\end{algorithm}

A few comments are in order.  Note that the test $\rho = \rho'$ is for equality,
not isomorphism.  In our implementation, the 
collections $Y_{\rho}$ are binary trees, rather than merely sets,
in order to speed up the search $\rho \in Y_{\rho}$.  

The task of finding all extensible partitions for a polymatroid $\rho$ is
relatively straightforward, but tedious.
First, a candidate for a modular cut $\mathcal{M}_0$
is found.  Since $\mathcal{M}_0$ is an up-closed set, it suffices to
keep track of the minimal flats in $\mathcal{M}_0$.  These are found as
independent sets of a graph with vertex set
equal to the flats of $\rho$.  If one flat is contained in another,
an edge is placed between the two.  Condition (2) of Theorem~\ref{thmC} is
then used to narrow the search. An edge is also placed between any two flats
that form a modular pair.
The independent sets in this graph are the minimal members of our candidates
for $\mathcal{M}_0$. Given an acceptable candidate for $\mathcal{M}_0$, a
more complicated procedure is used to find all possible candidates for $\mathcal{M}_1$.
The remaining flats are obviously assigned to $\mathcal{M}_2$.  Unfortunately,
the resulting partition must be checked to see if it satisfies conditions (1)
through (5), since some of these may fail for non-minimal members of $\mathcal{M}_0$
or $\mathcal{M}_1$.

Finally, note that each iteration of the outermost for loop may be run in
parallel since extensions of two different members of $X_n$ are never directly
compared to each other.

\section{Implementation and Results}

We implemented this algorithm in the C programming language.
In order to determine the cover relations for
flats, we employed the \texttt{ATLAS} library \cite{ATLAS} to multiply the adjacency
matrices of graphs. We used the \texttt{igraph} library \cite{igraph}
to find independent sets in graphs.  A computer
with a single 6-core Intel i7-3930K processor clocked at 3.20GHz 
running 64-bit Ubuntu Linux executed the resulting program.  After
approximately four days, a catalog of all $2$-polymatroids on seven or
fewer elements was generated.

The following table lists the number of $2$-polymatroids, up to isomorphism, on the ground set $\{1, \ldots, n\}$, by rank.

\begin{center}
\begin{tabular}{r||llllllll}
\multicolumn{1}{l}{} & \multicolumn{8}{c}{The number of unlabeled $2$-polymatroids} \\ 
\hline
rank $\backslash$ $n$ & 0 & 1 & 2 & 3 & 4 & 5 & 6 & 7 \\
\hline \hline
0   & 1 & 1 & 1 & 1  & 1   & 1     & 1     & 1             \\
1   &   & 1 & 2 & 3  & 4   & 5     & 6     & 7             \\
2   &   & 1 & 4 & 10 & 21  & 39    & 68    & 112           \\
3   &   &   & 2 & 12 & 49  & 172   & 573   & 1890          \\
4   &   &   & 1 & 10 & 78  & 584   & 5236  & 72205         \\
5   &   &   &   & 3  & 49  & 778   & 18033 & 971573        \\
6   &   &   &   & 1  & 21  & 584   & 46661 & 149636721     \\
7   &   &   &   &    & 4   & 172   & 18033 & 19498369      \\
8   &   &   &   &    & 1   & 39    & 5236  & 149636721     \\
9   &   &   &   &    &     & 5     & 573   & 971573        \\
10  &   &   &   &    &     & 1     & 68    & 72205         \\
11  &   &   &   &    &     &       & 6     & 1890          \\
12  &   &   &   &    &     &       & 1     & 112           \\
13  &   &   &   &    &     &       &       & 7             \\
14  &   &   &   &    &     &       &       & 1             \\
\hline
total & 1 & 3 & 10 & 40 & 228 & 2380 & 94495 & 320863387 \\
\hline
\end{tabular}
\end{center}

\bigskip
The following proposition is the key to producing the analogous table for
labeled $2$-polymatroids.
\begin{prop}The automorphisms of an integer polymatroid $(\rho, S)$ 
are in one-to-one correspondence with the automorphisms of its flat graph.
\end{prop}
\begin{proof}[Sketch of Proof]This is not hard to show.  It follows,
for example, from the remarks in Section 1 of \cite{millercs}, which
employs the language of hypergraphs.
\end{proof}

Since \texttt{nauty} can easily compute the automorphism groups of the 
flat graphs of these polymatroids, applying the Orbit-Stabilizer Theorem
gives a count of the number of labeled $2$-polymatroids on $7$ elements.
The following table lists the number of labeled $2$-polymatroids, on the
ground set $\{1, \ldots, n\}$, by rank.

\begin{center}
\begin{tabular}{r||llllllll}
\multicolumn{1}{l}{} & \multicolumn{8}{c}{The number of labeled $2$-polymatroids} \\ 
\hline
rank $\backslash$ $n$ & 0 & 1 & 2 & 3 & 4 & 5 & 6 & 7 \\
\hline \hline
0   & 1 & 1 & 1 & 1  & 1   & 1     & 1        & 1             \\
1   &   & 1 & 3 & 7  & 15  & 31    & 63       & 127           \\
2   &   & 1 & 6 & 29 & 135 & 642   & 3199     & 16879         \\
3   &   &   & 3 & 41 & 477 & 5957  & 87477    & 1604768       \\
4   &   &   & 1 & 29 & 784 & 27375 & 1554077  & 189213842     \\
5   &   &   &   & 7  & 477 & 41695 & 7109189  & 3559635761    \\
6   &   &   &   & 1  & 135 & 27375 & 21937982 & 733133160992  \\
7   &   &   &   &    & 15  & 5957  & 7109189  & 86322358307   \\
8   &   &   &   &    & 1   & 642   & 1554077  & 733133160992  \\
9   &   &   &   &    &     & 31    & 87477    & 3559635761    \\
10  &   &   &   &    &     & 1     & 3199     & 189213842     \\
11  &   &   &   &    &     &       & 63       & 1604768       \\
12  &   &   &   &    &     &       & 1        & 16879         \\
13  &   &   &   &    &     &       &          & 127           \\
14  &   &   &   &    &     &       &          & 1             \\
\hline
total & 1 & 3 & 14 & 115 & 2040 & 109707 & 39445994 & 1560089623047 \\
\hline
\end{tabular}
\end{center}

\bigskip

The symmetry of the columns in the above tables is explained by the following
notion of duality for $k$-polymatroids.
\begin{defin}
Given a polymatroid $(\rho, S)$, define the \emph{$k$-dual} 
$\rho^{*} \from 2^S \to \mathbb{N}$ by 
\begin{equation*}
\rho^{*}(X) = k|X| + \rho(S - X) - \rho(S).
\end{equation*}
\end{defin}
It is easily seen that $\rho^{*}$ is itself a $k$-polymatroid and that the
operation of $k$-duality is an involution on the set of $k$-polymatroids
on a fixed ground set which respects isomorphism.  (In fact, it is shown in 
\cite{whittlekdual} to be the the unique
such involution that interchanges deletion and contraction.)  

Welsh conjectured that the number of matroids
on a fixed set is unimodal in rank in \cite{welshcombprobs}. The 
counterpart of this conjecture for $k$-polymatroids is false.
The table above shows that it fails for $2$-polymatroids on $7$ elements.

Since the number of labeled $2$-polymatroids on seven elements is 
nearly a factor of $7!$ more than the number of unlabeled ones,
it seems reasonable to conjecture that, asymptotically, almost
all $2$-polymatroids are asymmetric. 

The proof in \cite{asymptoticconnected} that almost all matroids are loopless
carries over without change to $2$-polymatroids.  Our catalog suggests that 
a stronger property holds for $2$-polymatroids.  We conjecture that,
asymptotically, almost all $2$-polymatroids contain no elements of
rank less than $2$.  Here is the evidence from our catalog: 
the number of \emph{unlabeled} $2$-polymatroids on \{1, \ldots, $n\}$ 
with no elements of rank less than $2$.

\smallskip
\begin{center}
\begin{tabular}{r|lllllll}
\hline
$n$  & 1 & 2 & 3 & 4 & 5 & 6 & 7 \\
\hline
count & 1 & 2 & 8 & 51 & 696 & 49121 & 304541846 \\
\hline
\end{tabular}
\end{center}

\medskip
\noindent
This table should be compared to the first table in this section.

\smallskip

\section{A Confirmation}

Consider that the \emph{labeled} single-element extensions of a
$k$-polymatroid are in fact solutions to a certain integer programming problem. 
When all subsets of the ground set are taken as variables, inequalities guaranteeing
the axioms of a $k$-polymatroid are easily written.
To be concrete, let $\rho \from S \to \mathbb{N}$ be a $k$-polymatroid and
let $e$ be an element not in $S$.  Regard $\bar{\rho}(X)$ as a variable for 
each $X \subseteq S \cup e$. Fix $\bar{\rho}(A) = \rho(A)$ for 
$A \subseteq S$.  Also fix $\bar{\rho}(S\cup e) = \rho(S) + c$, where $c$ is
a natural number no greater than $k$.
Now nonnegative integer solutions to the system of inequalities below are
in one-to-one correspondence with labeled single-element extensions of $\rho$
which increase the rank of $\rho$ by $c$.

\begin{equation*}
\begin{split}
\bar{\rho}(A) + \bar{\rho}(A \cup f \cup g) \le \bar{\rho}(A \cup f) + \bar{\rho}(A \cup g)
& \quad \text{for} \: A \subseteq S \cup e \: \text{and} \: f, g \in (S \cup e) - A; \\
0 \le \bar{\rho}(A \cup f) - \bar{\rho}(A) \le k
& \quad \text{for} \: A \subseteq S \cup e \: \text{and} \: f \in (S \cup e) - A; 
\, \text{and} \\
\bar{\rho}(A) \le k|A| 
& \quad \text{for} \: A \subseteq S \cup e. \\
\end{split}
\end{equation*}

Here, we are using a condition equivalent to submodularity; see 
Theorem 44.2 of \cite{schrijverB} for a proof of equivalence.  The open-source optimization 
software SCIP \cite{SCIP} is able to count the number of integer solutions
to such inequalities.  Using SCIP we verified the numbers of labeled $2$-polymatroids 
given earlier.  Note that,
in the version of SCIP we used in April 2013, it was 
necessary to turn off all pre-solving options in order
to obtain accurate results.  This process took approximately 13 weeks
using the computer described in the previous section.

\smallskip

\section{Acknowledgements}
The author would like to thank the anonymous referee for carefully
reading this paper and providing suggestions which improved its exposition.

\smallskip

\bibliography{polymatroid_extensions}

\begin{thebibliography}{10}

\bibitem{SCIP}
Tobias Achterberg.
\newblock Scip: Solving constraint integer programs.
\newblock {\em Mathematical Programming Computation}, 1(1):1--41, July 2009.

\bibitem{eightels}
John~E. Blackburn, Henry~H. Crapo, and Denis~A. Higgs.
\newblock A catalogue of combinatorial geometries.
\newblock {\em Math. Comp.}, 27:155--166; addendum, ibid. 27 (1973), no. 121,
  loose microfiche suppl. A12--G12, 1973.

\bibitem{craposingleel}
Henry~H. Crapo.
\newblock Single-element extensions of matroids.
\newblock {\em J. Res. Nat. Bur. Standards Sect. B}, 69B:55--65, 1965.

\bibitem{igraph}
G{\'a}bor Cs{\'a}rdi and Tam{\'a}s Nepusz.
\newblock The igraph software package for complex network research.
\newblock {\em Inter{J}ournal Complex Systems}, page 1695, 2006.

\bibitem{mckayexposition}
Stephen~G. Hartke and A.~J. Radcliffe.
\newblock Mc{K}ay's canonical graph labeling algorithm.
\newblock In {\em Communicating Mathematics}, volume 479 of {\em Contemp.
  Math.}, pages 99--111. Amer. Math. Soc., Providence, RI, 2009.

\bibitem{lovaszplummer}
L{\'a}szl{\'o} Lov{\'a}sz and Michael~D. Plummer.
\newblock {\em Matching Theory}.
\newblock AMS Chelsea Publishing, Providence, RI, 2009.
\newblock Corrected reprint of the 1986 original.

\bibitem{matsumotoenum}
Yoshitake Matsumoto, Sonoko Moriyama, Hiroshi Imai, and David Bremner.
\newblock Matroid enumeration for incidence geometry.
\newblock {\em Discrete Comput. Geom.}, 47(1):17--43, 2012.

\bibitem{asymptoticconnected}
Dillon Mayhew, Mike Newman, Dominic Welsh, and Geoff Whittle.
\newblock On the asymptotic proportion of connected matroids.
\newblock {\em European J. Combin.}, 32(6):882--890, 2011.

\bibitem{nineels}
Dillon Mayhew and Gordon~F. Royle.
\newblock Matroids with nine elements.
\newblock {\em J. Combin. Theory Ser. B}, 98(2):415--431, 2008.

\bibitem{pgi1}
Brendan~D. McKay.
\newblock Practical graph isomorphism.
\newblock In {\em Proceedings of the {T}enth {M}anitoba {C}onference on
  {N}umerical {M}athematics and {C}omputing, {V}ol. {I} ({W}innipeg, {M}an.,
  1980)}, volume~30, pages 45--87, 1981.

\bibitem{pgi2}
Brendan~D. McKay and Adolfo Piperno.
\newblock Practical graph isomorphism, {II}.
\newblock {\em J. Symbolic Comput.}, 60:94--112, 2014.

\bibitem{millercs}
Gary~L. Miller.
\newblock Isomorphism of graphs which are pairwise {$k$}-separable.
\newblock {\em Inform. and Control}, 56(1-2):21--33, 1983.

\bibitem{oxleybook}
James Oxley.
\newblock {\em Matroid Theory}, volume~21 of {\em Oxford Graduate Texts in
  Mathematics}.
\newblock Oxford University Press, Oxford, second edition, 2011.

\bibitem{schrijverB}
Alexander Schrijver.
\newblock {\em Combinatorial optimization. {P}olyhedra and efficiency. {V}ol.
  {B}}, volume~24 of {\em Algorithms and Combinatorics}.
\newblock Springer-Verlag, Berlin, 2003.
\newblock Matroids, trees, stable sets, Chapters 39--69.

\bibitem{welshcombprobs}
D.~J.~A. Welsh.
\newblock Combinatorial problems in matroid theory.
\newblock In {\em Combinatorial {M}athematics and its {A}pplications ({P}roc.
  {C}onf., {O}xford, 1969)}, pages 291--306. Academic Press, London, 1971.

\bibitem{ATLAS}
R.~Clint Whaley and Antoine Petitet.
\newblock Minimizing development and maintenance costs in supporting
  persistently optimized {BLAS}.
\newblock {\em Software: Practice and Experience}, 35(2):101--121, February
  2005.

\bibitem{whittlekdual}
Geoff Whittle.
\newblock Duality in polymatroids and set functions.
\newblock {\em Combin. Probab. Comput.}, 1(3):275--280, 1992.

\end{thebibliography}
\bibliographystyle{plain}

\end{document}